\documentclass{amsart}
\usepackage{graphicx}

\newcommand{\e}{\varepsilon}
\newcommand{\LEE}{\mathcal{L}}
\newcommand{\R}{\mathbb{R}}
\newcommand{\C}{\mathbb{C}}
\newcommand{\Z}{\mathbb{Z}}
\newcommand{\F}{\mathcal{F}}

\theoremstyle{plain}
\newtheorem{thm}{Theorem}[section]
\newtheorem{lem}[thm]{Lemma}
\newtheorem{prop}[thm]{Proposition}

\theoremstyle{definition}
\newtheorem{defn}[thm]{Definition}
\newtheorem{exmp}[thm]{Example}
\newtheorem{pbm}[thm]{Problem}
\newtheorem{conj}[thm]{Conjecture}
\newtheorem*{rem}{Remark}

\newtheorem*{TB}{Thurston-Bennequin inequality}

\title[Reeb foliations on $S^5$ and contact $5$-manifolds]
{Reeb foliations on $S^5$ and contact $5$-manifolds
violating the Thurston-Bennequin inequality}

\author[A.~Mori]{Atsuhide MORI}
\address{Osaka city univerity advanced mathematical institute,
3-3-138 Sugimoto, Sumiyoshi-ku, Osaka~558-8585 Japan}

\email{ka-mori@ares.eonet.ne.jp}



\begin{document}

\begin{abstract}
This article describes the following results;  
i) convergence of high dimensional contact structure to 
codimension one foliation with Reeb component,
ii) relation between $Nil$-type and $Sol$-type contact submanifolds of $S^5$,
iii) definition of convex Thurston-Bennequin inequality and
iv) generalization of Lutz twist via convex hypersurface theory.
We perform Lutz twists along $Sol$-type submanifolds of $S^5$ 
to obtain exotic contact structures violating  
the inequality. We also describe the corresponding 
modification of foliation. 
\end{abstract}

\maketitle

\section{Introduction}
The author obtained the following results and 
distributed a preprint on them in 2009. 
Since some of the results has been used and brushed-up in literatures, 
especially in \cite{Kasuya} and \cite{Massot}, 
he would like to belatedly describe them in this article. 
We omit some details according to the current necessities.
\begin{enumerate}
\item[i)] In \S2, we construct a family of contact structures of
a closed manifold of dimension greater than three which converges 
to a codimension one foliation. The family is presented 
by a quadratic curve in the space of $1$-forms contrastingly 
to the linear deformation of $3$-dimensional contact structure to
a foliation in \cite{Mori}. 
In certain cases, this result has been generalized to a deformation 
of contact structure to leafwise symplectic foliation in \cite{MoriM}.
\item[ii)] In \S3, we briefly explain a contact geometrical relation between 
$Nil$-type and $Sol$-type submanifolds of the standard $S^5$, 
which has been completed by Naohiko Kasuya via 
complex singularity theory in \cite{Kasuya}.
\item[iii)] In \S4, we describe convex hypersurface theory due to Giroux  
and add to it a relative version. 
Then we define convex Thurston-Bennequin inequality. 
It is likely that any convex hypersurface (with boundary)  
in the standard $S^{2n+1}$ will satisfy the inequality, while there exists a
non-convex hypersurface in $S^{2n+1}$ which violates the inequality. 
(We describe it in Appendix, 
which has been another unpublished manuscript of this author.) 
\item[iv)] In \S5, we describe a generalization of Lutz twist via convex
hypersurface theory. This was generalized and improved to a satisfactory form 
by Massot, Niederkr\"uger and Wendl in \cite{Massot}. 
\end{enumerate}
We notice that the above results relate to each other. Indeed, in \S 6, we 
perform Lutz twists along certain $Sol$-type submanifolds of $S^5$ and 
obtain exotic (non-fillable) contact structures 
which violate the convex Thurston-Bennequin inequality. 
We also describe the corresponding modifications of foliations. 
This provides many examples of foliations of $S^5$ 
with (generalized) Reeb components.

\section{Convergence of contact structure to foliation}
We recall the definition of a supporting open-book structure in Giroux \cite{Giroux}.

\begin{defn} Let $\alpha$ be a contact form on a closed manifold $M^{2n+1}$,
and $N^{2n-1}$ a codimension two contact submanifold. Suppose that 
a tubular neighborhood of $N^{2n-1}$ is the product $N^{2n-1}\times D^2$,
and the pull-back of the angular coordinate of $D^2$ extends to 
a fiber bundle projection $\theta: M^{2n+1}\setminus N^{2n-1}\to S^1$. 
Then we say that the open-book structure defined by $\theta$ 
supports the contact structure $\ker\alpha$ if there exists a function 
$h: M^{2n+1}\to \R$ such that 
$d\theta\wedge(d(e^h\alpha))^n>0$ except on $N^{2n-1}$.
The contact form $\alpha_0=e^h\alpha$ is called an adapted contact form. 
\label{openbook}
\end{defn}  
A cooriented contact structure is defined by a contact form $\alpha$. 
If a supporting open-book structure is specified, 
we usually assume that $\alpha$ is adapted to it. 
Note that we can take the above $h$ so that 
the restriction $h|N^{2n-1}$ is arbitrary. 
Let $\rho$ denote the square of the radial coordinate of $D^2$.
Modifying $\rho$ and $\theta$ if necessary, 
we have the axisymmetric expression $\alpha=f_0(\rho)\beta+g_0(\rho)d\theta$
on $N^{2n-1}\times D^2$, where $\beta$ is 
(the pull-back of) the restriction $\alpha|N^{2n-1}$, 
$f_0(\rho)$ a positive decreasing function of $\rho$ with $f_0(0)=1$, and 
$g_0(\rho)$ a non-negative non-decreasing function of $\rho$ 
smoothly tangent to $\rho$ and $1$ respectively at $\rho=0$ and $\rho=1$. 
Our result of this section is 

\begin{thm} In the above setting for $n>1$, suppose that 
the contact submanifold $N^{2n-1}$ admits a non-zero 
closed $1$-form $\nu$ with $\nu \wedge (d\beta)^{n-1} = 0$.
Then there exists a family of contact forms 
$\{\alpha_t\}_{0\le t<1}$ on $M^{2n+1}$ 
which starts with $\alpha_0=\alpha$ and converges to a non-zero $1$-form 
$\alpha_1$ with $\alpha_1\wedge d\alpha_1=0$. 
That is, the contact structure $ker\alpha$ converges 
to a foliation defined by $\alpha_1$ at the end of an isotopic deformation.
\label{convergence}
\end{thm}

\begin{proof}
Take smooth functions $f_1(\rho)$, $g_1(\rho)$, $h(\rho)$ and $e(\rho)$ of 
$\rho\in [0,1]$ such that
\begin{enumerate}
\item[i)] $f_1= 1$ near $\rho=0$,\quad $f_1= 0$ on $[1/2,1]$,\quad
$f_1'\le 0$ on $[0,1]$,
\item[ii)] $g_1= 1$ near $\rho=1$,\quad $g_1= 0$ on $[0,1/2]$,\quad 
$g_1'\ge 0$ on $[0,1]$,
\item[iii)] $h= 1$ on $[0,1/2]$,\quad $h= 0$ near $\rho=1$,
\item[iv)] $e$ is supported near $\rho =1/2$,\quad and \quad $e(1/2)\ne 0$.
\end{enumerate}
Put $f_t(\rho)=(1-t)f_0(\rho)+tf_1(\rho)$, $g_t(\rho)=(1-t)g_0(\rho)+tg_1(\rho)$ and 
$$
\alpha_t|(N\times D^2)
=f_t(\rho)\{(1-t)\beta+th(\rho)\nu\}+g_t(\rho)d\theta+te(\rho)d\rho,
$$
where $\nu$ also denotes its pull-back. 
We extend $\alpha_t$ by putting
$$
\alpha_t|(M\setminus (N\times D^2))=\tau\alpha_0+(1-\tau)d\theta
\quad\textrm{where}\quad
\tau=(1-t)^2.
$$
We see from $d\nu= 0$ and $\nu \wedge (d\beta)^{n-1} = 0$ that
$\alpha_t\wedge (d\alpha_t)^n$ can be written as
$$
nf_t^{n-1}(1-t)^n(g_t'f_t-f_t'g_t)
\beta\wedge (d\beta)^{n-1}\wedge d\rho\wedge d\theta
$$
on $N\times D^2$ and 
$$
\tau^{n+1}\alpha_0\wedge(d\alpha_0)^n+\tau^n(1-\tau)
d\theta\wedge(d\alpha_0)^n
$$
on $M\setminus (N\times D^2)$.
Therefore we have
$$
\alpha_t\wedge (d\alpha_t)^n>0\quad (0\le t<1), \quad
\alpha_1\wedge d\alpha_1 = 0 \quad \textrm{and} \quad
\alpha_1\ne 0.
$$
This completes the proof of Theorem \ref{convergence}.
\end{proof}

If a cooriented foliation by oriented leaves has a minimal region 
bordered by closed leaves with totally coherent or totally anti-coherent orientation, 
we call it a dead-end component. 
Let $\F$ be the foliation defined by $\alpha_1$. Then $\F$ 
consists of two dead-end components, the Reeb component 
$\{\rho\le1/2\}\approx N^{2n-1}\times D^2$ and its complement. 
Here the orientation of the 
border leaf depends on the sign of $e(1/2)$.

\begin{rem}
A similar result in the case where $n=1$ is contained 
in the author's paper \cite{Mori}. It provides a foliation 
with Reeb component at the end of an isotopic
deformation of a given cooriented contact structure 
on a closed $3$-manifold. In this case 
the deformation can be presented by a straight line segment 
in the space of $1$-forms such that 
the Reeb field is transverse to the limit leaf along 
the core of the Reeb component for $0\le t<1$. 
The same result is also contained in Etnyre \cite{EtnyreSpinnable}.
Note that, contrastingly to this $3$-dimensional case, 
the above family $\{\alpha_t\}$ is not 
a straight line but a quadratic function of $t$, and 
the Reeb field keeps tangent to the limit leaf along the 
core $\{\rho=0\}$ of the Reeb component for $0\le t<1$.
\label{fol3}
\end{rem}

\section{Codimension two submanifolds}
Let $T_A$ denote the mapping torus 
$T^2\times \R/(\left(^x_y\right),z+1)
\sim (A\left(^x_y\right),z)$ 
of $A\in SL_2(\Z)$. In the case where $tr A\ge 2$, we have a 
natural contact form on $T_A$. Namely,

\begin{prop} 
1)\quad In the case of $A_{m,0}=
\left(
\begin{array}{cc}
1 & 0 \\
m & 1
\end{array}
\right)\in SL_2(\Z)$ for a positive integer $m$, 
$T_{A_{m,0}}$ is a $Nil$-manifold admitting the 
contact form $\beta=dy+mzdx$.\\
2) (Geiges \cite{Geiges}, 
Ghys \cite{Ghys}, Mitsumatsu \cite{Mitsumatsu})
Let $v_\pm$ be eigenvectors of 
$A$ with 
$$Av_\pm=a^{\pm 1} v_\pm,\quad \textrm{where}\quad 
a>1 \quad \textrm{and}\quad
dx\wedge dy(v_+,v_-)=1.$$
Then the $1$-forms $\beta_\pm=\pm a^{\mp z}dx\wedge dy(v_\pm,\cdot)$
define the Anosov foliations for the suspension flow. 
Using these forms we can construct 
the symplectic cylinder
$$([-1,1]\times T_A, d(\beta_++s\beta_-))\quad(s\in [-1,1])$$
with contact-type boundary
$(-T_A) \sqcup T_A$. Note that $T_A$ and $-T_A=T_{A^{-1}}$ are 
$Sol$-manifolds and the Anosov foliations are the Lie algebraic ones.
\label{1.2}
\end{prop}

Honda \cite{HondaII}, in a part of his classification 
of tight contact structures of $T^2$-bundles, 
showed that, for any $A\in SL_2(\Z)$ with $\mathrm{tr}\,A\ge 2$, 
there exist a positive integer $m$ and a $m$-tuple 
$k=(k_1,\dots,k_m)$ of non-negative integers such that the product
$$
A_{m,k}=
\left(
\begin{array}{cc}
1 & 0 \\
1 & 1
\end{array}
\right)
\left(
\begin{array}{cc}
1 & k_1 \\
0 & 1
\end{array}
\right)
\dots
\left(
\begin{array}{cc}
1 & 0 \\
1 & 1
\end{array}
\right)
\left(
\begin{array}{cc}
1 & k_m \\
0 & 1
\end{array}
\right)
$$
is conjugate to $A$ in $SL_2(\Z)$. Then
Van Horn-Morris \cite{VH} showed that
the above contact structure $\ker\beta$ or $\ker(\beta_-+\beta_+)$ 
of $T_{A_{m,k}}$ is supported by the open-book structure
determined up to equivalence by the following data:

\begin{description}
\item[Page] The page $P$, i.e., the fiber of $\theta$ in the previous section 
is the $m$ times punctured torus which is divided into $m$ copies $P_i$ 
($i=1,\dots,m$) of the three times punctured sphere along disjoint loops 
$\gamma_i$ between $P_i$ and $P_{i+1}$.
\item[Monodromy] The monodromy is the composition 
$\tau(\partial P)\circ \prod_{i=1}^m \{\tau(\gamma_i)\}^{k_i}$, 
where $\tau(\gamma)$ denotes the right-handed Dehn twist 
along a loop $\gamma$ on $P$, and $\tau(\partial P)$ 
the simultaneous right-handed Dehn twist along $\partial$-parallel loops. 
\end{description}
Using this description, the author proved the following theorem.

\begin{thm}
In the case where $m\le 2$, the above $T_{A_{m,k}}$ is 
contactomorphic to the link of the singularity $(0,0,0)$ 
of the complex surface 
$f_{m,k}(\xi,\eta,\zeta)=0$, where 
$$
f_{1,(k_1)}=\xi^2+(\eta-2\zeta^2)(\eta^2+2\zeta^2\eta+\zeta^4-\zeta^{4+k_1}) 
\quad \textrm{and}
$$
$$
f_{2,(k_1,k_2)}=\xi^2+\{(\zeta+\eta)^2-\zeta^{2+k_1}\}\{(\zeta-\eta)^2+\zeta^{2+k_2}\}.
$$
\label{main}
\end{thm}
Since we may consider the singularity links as contact submanifolds of $S^5$, 
this theorem provides a contact submanifold 
with $Nil$- or $Sol$-geometry according to $k=0$ or not. 
We call it a $Nil$-type or $Sol$-type contact submanifold. 
The author asked whether there exists a similar result for $m=3$ 
since $\ker\beta$ is still contactomorphic to a Brieskorn $Nil$-manifold
for $(m,k)=(3,0)$. Subsequently, 
Naohiko Kasuya \cite{Kasuya} gave a complete answer to this question. 
Namely,

\begin{thm}[Kasuya \cite{Kasuya}]
The above contact manifold $T_{A_{m,k}}$ is 
contactomorphic to a link of isolated surface singularity 
in $\C^3$ if and only if $m\le 3$, where the singularity can be taken as 
$(0,0,0)$ of the surface $\xi^a+\eta^b+\zeta^c+\xi\eta\zeta=0$
where $(a,b,c)$ stands for $(2,3,6+k_1)$ in the case where $m=1$, 
for $(2,4+k_1,4+k_2)$ in the case where $m=2$, and 
for $(3+k_1,3+k_2,3+k_3)$ in the case where $m=3$. 
\label{Kas}
\end{thm}

We would like to sketch the proof of Theorem \ref{main}. 
First we prepare an easy proposition 
which enabled the author to find the singularities. 

\begin{prop}
The complex curve $C: \xi^2=-(\eta-p_1)\cdots(\eta-p_{m+2})$ 
on the $\xi\eta$-plane is topologically 
a once or twice punctured, properly embedded, 
oriented surface with Euler characteristic $(-m)$
if the points $p_i$ are mutually distinct.
These points are the critical values of 
the hyperelliptic double covering $\pi_\eta|C$ where $\pi_\eta$ denotes the 
projection to the $\eta$-axis. 
Let $B: p_1=p_1(\theta),\dots,p_{m+2}=p_{m+2}(\theta)$ 
be a closed braid on $\C\times S^1$ $(\theta\in S^1)$. 
Then the above curve $C=C_\theta$ traces a surface bundle 
over $S^1$. 
Fix a proper embedding $l$ of $\R$ into the $\eta$-axis such that 
$$
l(1)=p_1(0),\dots,l(m+2)=p_{m+2}(0).
$$ 
Suppose that the closed braid $B$ is isotopic to 
the geometric realization of 
a word
$$
\prod_{j=1}^J\{\sigma_{i(j)}\}^{q(j)}\quad
(q(j)\in\Z,\,i(j)\in\{1,\dots,m+1\}),
$$
of the right-handed exchanges $\sigma_i:\C\to\C$ 
of strands $p_i$ and $p_{i+1}$ along the arc $l([i,i+1])$
$(i=1,\dots,m+1)$. 
Then the monodromy of the surface bundle $C_\theta$ is  
$$\displaystyle \prod_{j=1}^J\{\tau(\ell_{i(j)})\}^{q(j)}\quad
\textrm{where} \quad
\ell_i=(\pi_\eta|C)^{-1}(l([i,i+1])).$$
\label{easy}
\end{prop}

\begin{proof}[Sketch of the proof of {\rm Theorem \ref{main}}]
Regard $\zeta$ as a parameter and take
the branched double covering 
$\pi_\eta|C_\zeta$ of the curve $C_\zeta: f_{m,k}=0$ on the $\xi\eta$-plane.
Then the critical values of $\pi_\eta|C_\zeta$ are 
$$p_1,p_2=-\zeta^2\{1-(\zeta^{k_1})^{1/2}\}
\quad \textrm{and} 
\quad p_3=2\zeta_2\quad 
\textrm{if}
\quad m=1$$
$$(\textrm{resp.}\quad 
p_1,p_2=-\zeta\{1-(\zeta^{k_1})^{1/2}\},\,\,
p_3,p_4=\zeta\{1-(-\zeta^{k_2})^{1/2}\}\quad \textrm{if} \quad
m=2).$$
As $\zeta$ rotates along a small circle $|\zeta|=\e$ once counter-clockwise, 
the points $p_1(\zeta)$, \dots, $p_{m+2}(\zeta)$ traces a closed braid,
which is clearly a geometric realization of 
$$
(\sigma_1\circ\sigma_2)^6(\sigma_1)^{k_1}\qquad(\textrm{resp.}\quad
(\sigma_1\circ\sigma_2\circ\sigma_3)^4(\sigma_1)^{k_1}(\sigma_3)^{k_2}).
$$
From Proposition \ref{easy} and the relation
$$\tau(\partial C_\zeta)\simeq 
(\tau(\ell_1)\circ \tau(\ell_2))^6
\quad (\textrm{resp.}\quad
\tau(\partial C_\zeta)\simeq 
(\tau(\ell_1)\circ \tau(\ell_2) \circ \tau(\ell_3))^4),$$
we see that the singularity link of $\{f_{m,k}=0\}$ 
admits the above mentioned open-book structure. 
To be more precise, we have to consider everything in a small ball, 
take $\e$ much smaller, and use the result of Caubel-Nemethi-Popescu-Pampu \cite{CNP}. \end{proof}

\section{The Thurston-Bennequin inequality for convex hypersurfaces}

\subsection{Preliminaries on the Thurston-Bennequin inequality}
Let $\Sigma$ be a compact connected oriented hypersurface 
embedded in a contact $(2n+1)$-manifold $(M^{2n+1},\ker\alpha)$, and 
$S_+(\Sigma)$ (resp. $S_-(\Sigma)$) the set of the positive (resp. negative) 
tangent points of $\Sigma$ to $\ker\alpha$. Here
the sign of the tangency coincides with 
the sign of $(d\alpha|\Sigma)^n$ at each tangent point. 
We assume that 
the union $S(\Sigma)=S_+(\Sigma)\cup S_-(\Sigma)$ 
is a discrete subset in $\textrm{int}\,\Sigma$. 
Considering on the symplectic hyper\-plane
$(\ker\alpha,d\alpha|\ker\alpha)$ at each point on $\Sigma$,
we see that the symplectic orthogonal of the intersection
$T\Sigma\cap \ker\alpha$ forms an oriented line field 
$L$ on $\Sigma$. 
Then the set of singular points of $L$ 
coincides with the set $S(\Sigma)$ of tangent points.

\begin{defn}
The singular foliation $\F_\Sigma$ defined by $L$ is
called the characteristic foliation of $\Sigma$ 
with respect to the contact structure $\ker\alpha$.
\end{defn}

\noindent
Put $\lambda=\alpha|\Sigma$ and take any volume form $dvol$
on $\Sigma$. Then we see that the vector field $X$ on 
$\Sigma$ defined by 
$\iota_X dvol=\lambda\wedge(d\lambda)^{n-1}$ 
is a positive section of $L$. 
Moreover
$$
\iota_X\{\lambda\wedge(d\lambda)^{n-1}\}=
-\lambda\wedge\iota_X(d\lambda)^{n-1}=0\quad
(\lambda\wedge(d\lambda)^{n-1}\ne 0).
$$
This implies that $\lambda\wedge\iota_Xd\lambda=0$, and therefore the flow of $X$
preserves the conformal class of $\lambda$.
Since $dvol$ is arbitrary, we may take $X$ as any positive section of $L$.
Therefore $\lambda$ defines a holonomy invariant
transverse contact structure of $\F_\Sigma$.
Note that, even for another volume form $dvol'$ on $\Sigma$, 
the sign of $\mathrm{div}\, X=(\LEE_X dvol')/dvol'$ at each singular point $p\in S(\Sigma)$ 
coincides with the sign of the tangency at $p$.
Thus $\F_\Sigma$ itself contains
the information about the sign of the
tangency. 

On the other hand, by using a positive section $X$ of $L$, 
we can define the index
$\mathrm{Ind}\,p$ of a singular point $p\in S(\Sigma)$
as the vector field index $\mathrm{Ind}_X p$.

\begin{defn}
Suppose that the boundary $\partial \Sigma$ 
of the above hypersurface $\Sigma$ is non-empty, 
and the characteristic foliation $\F_\Sigma$ is 
positively (i.e., outwards) transverse to $\partial\Sigma$. 
Then we say that $\Sigma$ is a hypersurface with
contact-type boundary. The
contact-type boundary $\partial\Sigma$ inherits 
the contact form $\lambda|\partial\Sigma=\alpha|\partial\Sigma$.
\end{defn}

\begin{rem}
We usually fix a primitive $1$-form $\lambda$ 
on an exact symplectic manifold $(\Sigma, d\lambda)$. 
This is equivalent to fix a vector field $X$ 
with $\iota_X d\lambda=\lambda$.
If $X$ is 
positively transverse to the boundary $\partial \Sigma$,
then $(\partial\Sigma,\lambda|\partial\Sigma)$ is called the 
contact-type boundary.
The above definition is a natural shift of this notion into 
our setting.
\end{rem}

Let $D^2$ be an embedded disk with contact-type boundary
in a contact $3$-manifold. We say that $D^2$ is overtwisted
if the singularity $S(D^2)$ consists of a single sink point.
Note that a sink point is a negative singular point
since it has negative divergence.
If a contact $3$-manifold contains overtwisted disks, 
we say that it is overtwisted, else it is tight.
We can show that the existence of an overtwisted disk
with contact-type boundary
is equivalent to the existence of an overtwisted disk with 
Legendrian boundary, which is an embedded disk $D'$ similar to the above 
$D^2$ except that the characteristic foliation $\F_{D'}$
is tangent to the boundary $\partial D'$, where $\partial D'$ (or $-\partial D'$)
is a closed leaf of $\F_{D'}$. Indeed perturbing the above $D^2$ if necessary, 
we can find $D'$ in $D^2$. Conversely, the characteristic foliation of 
the boundary of an usual neighborhood of $D^2$ 
contains a sink point $N$ (north pole), a source point $S$ (south pole)
and a pair of closed orbits which bounds a tubular neighborhood of the equator. 
Thus, removing a narrow tubular neighborhood of 
the closed orbit with negative divergence, 
we obtain the above $D^2$ as the disk containing $N$.  

Let $\Sigma$ be any surface with contact-type (i.e., positive
transverse) 
boundary embedded in the standard $S^3$.
Then Bennequin \cite{Bennequin} proved the following inequality
which implies the absence of overtwisted disks in $S^3$, i.e., the tightness of $S^3$.

\begin{TB}[I]
$\displaystyle \sum_{p\in S_-(\Sigma)}\mathrm{Ind}\,p \leq 0$.
\end{TB}

\noindent
Eliashberg proved the same inequality for symplectically fillable
contact $3$-manifolds (\cite{Eliashberg2}),
and finally for all tight contact $3$-manifolds (\cite{Eliashberg3})
by applying the elimination lemma in Giroux \cite{GirouxConvex}.
Niederkr\"uger \cite{Niederkruger} and Chekanov found
an $(n+1)$ dimensional analogue of an overtwisted disk,
i.e., a plastikstufe (or an overtwisted family)
which is the product $D^2\times L^{n-1}$ of an overtwisted disk $D^2$
with a closed isotropic submanifold $L^{n-1}$ (see \S 5.2 for the precise definition).
However, in order to create some meaning of the above inequality
in higher dimension, we need a $2n$-dimensional analogue of 
overtwisted disk. Here we should notice that a folklore says that 
the above inequality does not hold in higher dimension 
as a merely algebraic inequality described in the following remark (see also Appendix). 
In order to bring out the geometric flavor of the inequality, 
we put a strong limit on the test hypersurfaces. 
It is the ``convexity'' in the next subsection. 
\begin{rem}
The Thurston-Bennequin inequality can also be written in terms of
relative Euler number as follows. The above vector field $X\in T\Sigma \cap \ker\alpha$
is a section of $\ker\alpha|\Sigma$ which is canonical (i.e. exhausting)
near the boundary $\partial\Sigma$. Thus under a certain 
coherent boundary condition we have
$$
\langle e(\ker\alpha),\,[\Sigma,\partial\Sigma]\rangle=
\sum_{p\in S_+(\Sigma)}\mathrm{Ind}\,p-\sum_{p\in S_-(\Sigma)}\mathrm{Ind}\,p.
$$ 
Then the Thurston-Bennequin inequality can be expressed as 
$$
-\langle e(\ker\alpha),\,[\Sigma,\partial\Sigma]\rangle \leq -\chi(\Sigma).
$$
There is also an absolute version of the inequality
for a closed hypersurface $\Sigma$ with $\chi(\Sigma)\leq 0$ expressed as
$|\langle e(\ker\alpha),\,[\Sigma]\rangle|\leq -\chi(\Sigma)$,
or equivalently as
$$
\sum_{p\in S_-(\Sigma)}\mathrm{Ind}\,p \leq 0\quad\quad
\textrm{and}\quad\quad
\sum_{p\in S_+(\Sigma)}\mathrm{Ind}\,p \leq 0.
$$
The absolute version trivially holds
if the Euler class $e(\ker \alpha)\in H^{2n}(M^{2n+1}; \Z)$
is a torsion.
Note that the inequality and its absolute version can be defined
for any oriented plane field on an oriented $3$-manifold $M^3$
(see Eliashberg-Thurston \cite{EliashbergThurston}).
They are originally proved for a foliation on $M^3$ without Reeb components
by Thurston (see \cite{Thurston}). On the other hand, 
we can deform even a tight contact structure to 
a foliation with Reeb component (\cite{Mori}; see also \cite{MM}).
Thus many foliations with Reeb components also satisfy the inequality.   
\end{rem}
\subsection{Convex hypersurfaces}
In this subsection, we review the convex hypersurface theory 
doe to Giroux \cite{GirouxConvex}, and add to it a relative version.

A vector field $Z$ on a contact manifold 
$(M^{2n+1}, \ker\alpha)$ is called a contact vector field if
it satisfies $\alpha\wedge\LEE_Z\alpha=0$, i.e., preserves $\ker\alpha$. 
Let $V_{\ker\alpha}$ denote the space of contact vector fields 
for $\ker\alpha$. It is fundamental that the contact form $\alpha$ 
determines the linear isomorphism
$\alpha(\cdot):V_{\ker \alpha}\to C^\infty(M^{2n+1})$.
Indeed we can identify $V_{\ker \alpha}$ to the space of Hamiltonian vector 
fields on the symplectization $\left(\R(\ni s)\times M^{2n+1}, d(e^s\alpha)\right)$ 
for functions which split as $e^sH$ ($H\in C^\infty(M^{2n+1})$). 
\begin{defn}
1)
For a contact vector field $Z$, 
the function $H=\alpha(Z)$ is called
the contact Hamiltonian function. 
Conversely for a function $H$ on $M^{2n+1}$,
the contact vector field $Z$ uniquely determined by $\alpha(Z)=H$ is called
the contact Hamiltonian vector field.
The contact Hamiltonian vector field of the constant function $1$ lies in the degenerate 
direction of $d\alpha$ and is called the Reeb field.

2)
A closed oriented hypersurface $\Sigma$ embedded in a contact manifold
is said to be convex if there exists a contact vector field
transverse to $\Sigma$.
\end{defn}

Let $Z$ be a contact Hamiltonian vector field of a function $H$ 
which is positively transverse to a closed convex hypersurface $\Sigma$. 
Perturbing $H$ if necessary, we may assume that 
the level set $\{H=0\}$ is a regular hypersurface transverse to $\Sigma$. 
Let $\Sigma\times(-\e,\e)$ denote a neighborhood of 
$\Sigma$ on which $Z$ can be expressed as $\partial/\partial z$ ($z\in(-\e,\e)$).
In the rest of this subsection, we restrict everything to 
$\Sigma\times(-\e,\e)$. Further we assume that
$\alpha$ is $z$-invariant by multiplying it by a positive function if necessary.
Then the restrictions 
$\pm Z|\{\pm H>0\}$ are the Reeb fields of 
$\pm \alpha/H$ and $Z$ is tangent to their partition $\{H=0\}$. 
Indeed we can write $\alpha=\lambda+ H dz$ where $\lambda$ is
(the pull-back of) the restriction $\alpha|\Sigma$ 
and $H$ is a $z$-invariant function.

\begin{defn} The above
$\Gamma$ is called the dividing set on $\Sigma$ with respect to $Z$.
$\Gamma$ divides
$\Sigma$ into the positive region
$\Sigma_+=\{(H|\Sigma)\ge0\}\subset \Sigma$ and the negative region
$-\Sigma_-=\{(H|\Sigma)\le0\}\subset \Sigma$. 
We orient $\Gamma$ as $\Gamma=\partial \Sigma_+$ 
or equivalently as $\Gamma=\partial \Sigma_-$.
\end{defn}

We see that the $2n$-form 
$\Omega=(d\lambda)^{n-1}\wedge (Hd\lambda+n\lambda dH)$ 
satisfies $\Omega\wedge dz=\alpha\wedge(d\alpha)^n$. 
Thus the characteristic foliation $\F_\Sigma$ is positively transverse 
to the dividing set $\Gamma$. Since $\lambda$ defines the 
holonomy invariant transverse contact structure of $\F_\Sigma$, 
$\Gamma$ is a contact submanifold. Let $\beta$ be any
contact form presenting the contact structure.
Then, changing $\alpha$ and $H$ with keeping $\ker\alpha$ and $\Gamma$ 
if necessary, we have $\alpha=e^{-H^2}\beta+Hdz$  
on a small neighborhood $\{-\e'<H<\e'\}$ of $\Gamma\times(-\e,\e)$.

Let $d\lambda_\pm$ be exact symplectic forms 
on $\textrm{int}\,\Sigma_\pm$ such that $\lambda_\pm\wedge\lambda=0$
and $\lambda_+|\Gamma=\lambda_-|\Gamma=\beta$.
Then $\lambda_\pm\pm dz$ is a $z$-invariant contact form on 
$\textrm{int}\,\Sigma_\pm \times(-\e,\e)$.

\begin{defn}
1) Let $(\textrm{int}\,\Sigma_+,d\lambda_+)$ be an exact symplectic manifold. 
Then the pair $(\textrm{int}\,\Sigma_+\times\R,\lambda_++dz)$ ($z\in \R$) 
is called its contactization with respect to $\lambda_+$.

2) Let $(\Sigma_+,d\lambda_+)$ be the compactification of 
an exact symplectic manifold with contact-type end. 
Precisely, assume that the primitive $1$-form $\lambda_+$   
can be expressed as $e^{-s^2}\beta$ on a collar neighborhood 
$(-\e',0]\times\partial\Sigma_+$ of the manifold $\Sigma_+$,
where $s$ is the coordinate of $(-\e',0]$, 
and $\beta$ is the pull-back of $\lambda_+|\partial\Sigma$.  
Take a function $g_+$ on $\Sigma_+$ such that 
$g_+$ is a decreasing function of $s$ on $(-\e',0]\times \partial\Sigma_+$,
$g_+=-s$ holds near $\partial\Sigma_+$, and $g_+=1$ holds except on the collar.
Then $\alpha_+=\lambda_++g_+dz$ is a contact form on 
the cylinder $\Sigma_+\times \R$. 
We call $(\Sigma_+\times\R,\ker\alpha_+)$
the modified contactization of 
the exact symplectic manifold $(\textrm{int}\,\Sigma_+,d\lambda_+)$ 
with respect to $\lambda_+$. 
\label{modified}
\end{defn}

We can identify the boundaries of the modified contactizations 
$\Sigma_+\times\R$ ($z\in\R$) w.r.t. $\lambda_+$ 
and $\Sigma_-\times(-\R)$ w.r.t. $\lambda_-$
as is depicted in Figure \ref{unify}. 
\begin{figure}[ht]
\centering
\includegraphics{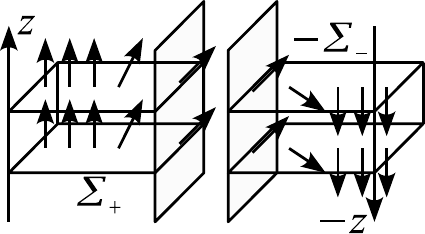}
\caption{The modified
contactizations: The arrows present the Reeb fields of $\alpha_\pm$
and the shaded parts are identified.}
\label{unify}
\end{figure}
Then we 
obtain a contact form $\widetilde{\alpha}=\alpha_\pm$ on 
$(\Sigma_+\cup(-\Sigma_-))\times\R$ which is expressed as 
$\widetilde{\alpha}=e^{-s^2}\beta+ g dz$ on $(-\e',\e')\times \partial\Sigma_+\times\R$, 
where $g=\pm g_\pm$ on $\Sigma$ is the smoothing of $\pm1$ on $\pm\Sigma_\pm$.

\begin{defn}
The contact manifold $((\Sigma_+\cup(-\Sigma_-))\times \R,\ker\widetilde{\alpha})$
is called the unified contactization of $\Sigma=\Sigma_+\cup(-\Sigma_-)$.
\label{unified}
\end{defn}

Since the convex hypersurface $\Sigma\subset (M^{2n+1},\alpha)$ is compact, 
the original neighborhood $\Sigma\times(-\e,\e)$ is contactomorphic to 
a neighborhood of $(\Sigma_+\cup(-\Sigma_-))\times\{0\}$ in the unified contactization. 
This is the convex hypersurface theory due to Giroux. We add to it a relative version. 

\begin{defn}
Let $\Sigma$ be a compact connected oriented hypersurface with 
non-empty contact-type boundary embedded in a contact manifold $(M^{2n+1},\alpha)$.
Then $\Sigma$ is said to be convex if there exists a contact vector field $Z$ 
which is positively transverse to $\Sigma$ and satisfies the boundary condition 
$\alpha(Z)|\partial\Sigma>0$.
\label{convexb}
\end{defn}

We also put $H=\alpha(Z)$ after suitably perturbing 
the contact vector field $Z$. Then 
the dividing set $\Gamma=\{H=0\}\cup \Sigma$
divides $\Sigma$ into the positive region $\Sigma_+$ 
and the (possibly empty) negative region $-\Sigma_-$ 
in the same way as above. 
Then we have 
$\partial\Sigma=\partial \Sigma_+ \setminus \partial \Sigma_-\ne\emptyset$
(see Figure \ref{convex}).

\begin{figure}[h]
\centering
\includegraphics{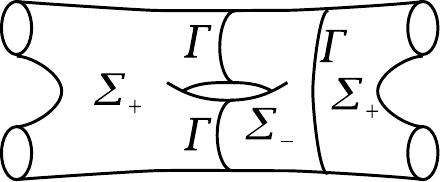}
\caption{A convex hypersurface with contact-type boundary}
\label{convex}
\end{figure}

We can also construct the modified contactization and use it as the model 
of a neighborhood of the convex hypersurface as long as 
we do not modify the contactization of $\Sigma_+$ near 
$\partial\Sigma\times \R\subset \partial\Sigma_+\times \R$.
We explain the treatment of the boundary $\partial\Sigma$ 
in the following subsection. 
\subsection{Open-book structure with convex pages}
We can construct a contact structure from a given open-book structure 
with convex pages. Then the open-book structure 
can be considered as a generalization 
of ``quasi-compatible'' open-book structure introduced by 
Etnyre and Van Horn-Morris in \cite{EV}. 
In fact Etnyre commented that their result on Giroux 
torsion might be generalized to the setting of this article.  
His idea motivated the work of 
Massot, Niederkr\"uger and Wendl in \cite{Massot} as
is mentioned there. The idea of placing a 
Lutz tube along the binding of an open-book structure 
is also found in Ishikawa's work \cite{Ishikawa}. 

\begin{prop}[Construction of open-book structure with convex pages]
Let $(\Sigma_\pm,d\lambda_\pm)$ be two compact exact symplectic manifolds
with contact-type boundary. 
Suppose that there exists an inclusion $\iota:\partial\Sigma_-\to
\partial\Sigma_+$ preserving the contact form, i.e.,  
$\iota^*(\lambda_+|\partial\Sigma_+)=\lambda_-|\partial \Sigma_-$.
Similarly to the construction of unified contactization, 
we modify the exact symplectic structure so that 
they match up to define an exact $2$-form $d\lambda$
on the union $\Sigma=\Sigma_+\cup_\iota (-\Sigma_-)$. 
Precisely, the modification is supported in a small neighborhood 
$(-2\e', 2\e')\times \Gamma$ of the dividing set 
$\Gamma=\partial(-\Sigma_-)\subset\Sigma$, and 
the primitive $1$-form $\lambda$ is locally expressed as 
$\lambda=e^{-s^2}\beta$ for $-\e'<s<\e'$, 
where $s$ denotes the coordinate of the interval $(-2\e', 2\e')$, and
$\beta$ (the pull-back of) the above $\iota$-invariant 
contact form on $\Gamma$. 
Let $\varphi: \Sigma\to\Sigma$ be a diffeomorphism such that
\begin{enumerate}
\item[i)] the support of $\varphi$ does not intersect with 
$((-2\e',2\e')\times\Gamma)\cup\partial\Sigma$, and
\item[ii)] $\varphi^*\lambda-\lambda=\pm dh$ 
(according to the sign of the region) holds for a suitable 
positive function $h$ on $\Sigma$ 
which is equal to $1$ on $((-2\e',2\e')\times\Gamma)\cup\partial\Sigma$. 
\end{enumerate}
Then the contact form $\widetilde{\alpha}$ 
of the unified contactization $\Sigma\times\R$ 
determines a contact form on the mapping torus
$\Sigma\times\R/(x,z+h)\sim(\varphi(x),z)$. 
We cap-off the boundary $\partial\Sigma\times(\R/\Z)$ by the tube
$$
\left(\partial\Sigma\times D^2, \ker\left(
\frac{f_0(\rho)}{f_0(1)}\lambda|\partial\Sigma
+\frac{hg_0(\rho)}{2\pi}d\theta\right)\right),
$$
where $f_0(\rho)$ 
and $g_0(\rho)$ are the functions in the setting of Theorem \ref{convergence}. 
This provides a closed contact manifold $(M^{2n+1},\ker\alpha)$
on which the family of the convex hypersurfaces 
$\{\theta=2\pi z/h=\textrm{const}\}$ 
defines an open-book structure.  
\label{openbook2}
\end{prop}

The hypersurface $\{\theta\equiv 0 \mod \pi\}$ is clearly convex. 
In other word, the page $\Sigma'=\{\theta=0\}$ is an extension of 
$\Sigma$ such that $\partial\Sigma'$ is contained in the dividing set of $\Sigma'$. 
We can define the unified contactization of $\Sigma'$ just by partially glueing 
the modified contactizations in Definition \ref{modified} 
as is described in Definition \ref{unified}. 
Clearly, the unified contactization of $\Sigma$ is a contact submanifold 
of the extended unified contactization of $\Sigma'$.
The construction in Proposition \ref{openbook2} 
can be considered as follows. 
First we extend the mapping torus of $\Sigma$ 
to that of $\Sigma'$. Then we shrink 
its boundary $\partial\Sigma'\times \R/\Z$ 
into $\partial\Sigma'$ to obtain the open-book structure.
\begin{defn}
We say that one unbinds the open-book structure 
to the mapping torus and 
one rebinds the latter to the former. 
We use the same terminologies even when 
the open-book structure is defined only near the binding. 
\label{unbind}
\end{defn}

\begin{rem} 
Giroux proved that any symplectomorphism 
supported in $\mathrm{int}\,\Sigma_+$ is isotopic 
through such symplectomorphisms to
$\varphi$ with $\varphi^*\lambda_+-\lambda_+=dh_+$ ($\exists h_+>0$).
From contact topological point of view, 
since we can scale anything at will, 
we may assume that $h|\partial\Sigma_+=1$.  
It is remarkable that he can prove the existence of 
supporting open-book decomposition. Namely, 
we can obtain a supporting open-book structure
of a given closed contact manifold by using 
the result of Ibort, Mart\'inez-Torres and Presas in \cite{IMP}
on applicability of the Donaldson-Auroux approximately 
holomorphic method 
to complex valued functions on contact manifolds
(see \cite{Giroux}). 
Trivially, this also implies the existence of the above open-book structure.   
\end{rem}

\subsection{The inequality for convex hypersurfaces}
For a convex hypersurface $\Sigma$ with contact-type boundary, 
the inequality in \S4.1 can also be written as 

\begin{TB}[II]
$\displaystyle \chi(\Sigma_-)\leq 0$
{\rm (}or $\Sigma_-=\emptyset${\rm )}.
\end{TB}

\noindent
Suppose that there exists a convex disk $\Sigma$ with
contact-type boundary in a contact $3$-manifold which is the union
$\Sigma_+\cup(-\Sigma_-)$ of a negative disk region $\Sigma_-$ and
a positive annular region $\Sigma_+$.
Then, since $\chi(\Sigma_-)=1>0$, 
the convex disk $\Sigma$ violates the Thurston-Bennequin inequality.
We call the convex disk $\Sigma$ a convex overtwisted disk.
The Giroux approximation in \cite{GirouxConvex}, 
implies that any overtwisted disk with contact-type boundary is 
approximated by a convex overtwisted disk.
Moreover, it also implies that the inequality for 
any convex surface holds if and only if 
the contact $3$-manifold is tight. 
This leads us to the following definition.

\begin{defn}
A convex overtwisted hypersurface 
is a connected convex hypersurface $\Sigma_+\cup(-\Sigma_-)$ with 
non-empty contact-type boundary
which satisfies $\chi(\Sigma_-)>0$.
We say that a contact structure is convex-overtwisted (resp. convex-tight)
if it contains some (resp. no) overtwisted convex hypersurface.
\end{defn}

\noindent
Allegorically, a convex hypersurface percepts `extra parts' in a 
contact manifold
at the sensor $\Sigma_-$
through the counter $\chi(\Sigma_-)$ 
where each extra part must stick out a `tight' wear (i.e. tight contact structure). 
In other words, we see that the manifold 
is wearing a `loose' (i.e. overtwisted) contact structure if 
the number $\chi(\Sigma_-)$ is positive. 
Indeed the original Thurston-Bennequin inequality
expresses the tightness of a contact $3$-manifold 
as the absence of `extra parts' (i.e. Lutz tubes).

\begin{rem}
Any convex overtwisted hypersurface 
$\Sigma$ must contain a connected component $S_+$ of 
the positive region $\Sigma_+$ with 
$\partial S_+\cap \partial \Sigma\neq\emptyset$
and $\partial S_+\cap\partial \Sigma_-\neq\emptyset$. 
Thus $S_+$ is a connected symplectic manifold
with disconnected contact-type boundary. 
The existence problem of such a symplectic manifold is
called Calabi's question, and McDuff \cite{McDuff} 
found the first example. 
Then a simpler example appeared in the 
literatures referred in Proposition \ref{1.2} 2). 
It is a cylinder with $3$-manifold base presenting 
a gradual exchange between positive and negative contact structures. 
Mitsumatsu \cite{Mitsumatsu} further studied on 
coexistence situation of positive and negative contact structures 
by means of a generalization of Anosov flow (i.e. projectively Anosov flow). 
On the other hand a Lutz twist can be considered as a local 
exchange of the direction of a closed orbit $K$ of a Reeb flow.
Here the direction is determined by 
the sign of the restricted contact form on $K$. 
These facts motivated this author to generalize Lutz twist. 
Let $N^3$ be a codimension two oriented submanifold 
of a contact $5$-manifold $(M^5,\ker\alpha)$ with trivial normal bundle. 
Suppose that $N^3$ is tangent to the Reeb flow of $\alpha$, 
and the restriction $\alpha|N^3$ is a positive contact form.
Further suppose that an Anosov flow on $N^3$ 
presents an exchange of $\alpha|N^3$ with a negative contact form. 
Then, as is described in the next section, we can 
make a modification of the contact structure supported near $N^3$ 
which realizes that exchange. 
Note that the Anosov flow on the submanifold
tangent to both of the positive and negative contact structures. 
In the original $3$-dimensional case, the contact structures are 
the oriented zero sections of $TK$, and we may consider that 
the $0$-dimensional flow is tangent to them. 
In other word, the notion of Anosov flow on $3$-manifold 
generalizes the identity on the circle.
\end{rem}
 
\section{A generalization of Lutz twist}
\subsection{Definitions and examples}
We will perform a generalization of Lutz twist along a twistable contact knot
defined as follows.

\begin{defn}
A codimension two closed contact submanifold 
$(K^{2n-1},\ker\beta)$ of a contact manifold $(M^{2n+1},\ker\alpha)$ is 
called a contact knot if it is connected and its normal disk bundle is trivial. 
Suppose that there exists a cylinder $I \times K^{2n-1}$ which is the
compactification of an exact symplectic manifold with contact-type end.
Here we assume that the end $\{1\}\times K^{2n-1}$ is contactomorphic 
to $(K^{2n-1}, \ker\beta)$. Then we say that the contact knot 
$K^{2n-1}\subset M^{2n+1}$ is twistable.
\end{defn}

The contact submanifold $N^{2n-1}$ in 
Definition \ref{openbook} (or $\partial\Sigma$ in 
Proposition \ref{openbook2}) is a contact knot 
if it is connected. Now we assume that $K^{2n-1}=N^{2n-1}$ is twistable. 
Then we firstly unbind the contact manifold $M^{2n+1}$ 
along $K^{2n-1}$ as is described in Definition \ref{unbind}; 
secondly insert the quotient $I\times K^{2n-1}\times (\R/\Z)$ of the 
modified contactization of the cylinder $I\times K^{2n-1}$ 
with turning it upside down, i.e., as $(-I\times K^{2n-1}) \times (-\R/\Z)$; 
and lastly rebind 
the new boundary $(-K^{2n-1})\times (-\R/\Z)=K^{2n-1}\times (\R/\Z)$ 
to regain a closed contact manifold diffeomorphic to $M^{2n+1}$. 
This modification is topologically the insertion of the tube 
$K^{2n-1}\times D^2$ along $K^{2n-1}$. From open-book point of view, 
it is the addition of exact symplectic collar $I\times K^{2n-1}$ 
to the page with reversed orientation. 
Note that the reversion of the orientation of the convex 
hypersurface with non-empty contact-type boundary 
makes the sign of the region next to the boundary 
negative, and therefore breaks the boundary condition
in Definition \ref{convexb}. 
Thus the addition of the positive collar is mandatory. 
Since it does not change the manifold $M^{2n+1}$,
it is considered as a modification of the contact structure. 
Moreover, since the contact structure 
around any contact knot $(K^{2n-1},\ker\beta)$ can be written as
$$\ker(e^{-\rho}\beta+\rho d\theta)\quad ((\sqrt{\rho},\theta)\in D^2,\,\,\, \rho\ll1),$$
we can also unbind the contact manifold $(M^{2n+1},\ker\alpha)$ along 
$K^{2n-1}$ to obtain a portion of modified contactization. 
(This is just a local construction. Indeed, 
if $(K^{2n-1},\beta)$ is not fillable, 
we can not construct a whole modified contactization.)
Thus we can perform a similar insertion of tube along any twistable contact knot.  

\begin{prop}[Definition of a generalization of Lutz twist]
We call the above tube $K^{2n-1}\times D^2$ an abstract Lutz tube. 
To show its existence, 
we describe explicitly the cylinder $I\times K^{2n-1}$ and the 
other end $-\{0\}\times K^{2n-1}$, which become
the page and the binding of the open-book structure of the Lutz tube. 
Suppose that 
a contact knot $K^{2n-1}$ with contact form $\beta$ also 
admits a $1$-form $\mu$ such that
\begin{enumerate}
\item[i)] $\mu\wedge (d\mu)^{n-1}
=-\beta\wedge (d\beta)^{n-1}(<0)$\quad and
\item[ii)] the other $(2n-1)$-forms presented 
by products of $\beta, d\beta, \mu$ and $d\mu$ vanish.
\end{enumerate}
Then we see that the $1$-form
$$
\alpha'=\sin^2\frac{\pi \rho}{2}\beta
+\cos^2\frac{\pi \rho}{2}\mu
-\sin (\pi \rho) d\theta
$$
defines the contact structure of the tube 
$(-K^{2n-1})\times (-D^2)$
whose core $-K^{2n-1}$ inherits the contact form $\mu$.
where $(\sqrt{\rho},\theta)$ is the polar coordinates of $D^2$.
We call it the (half) Lutz tube with core $(-K^{2n-1},\ker\mu)$ and
the longitude $(K^{2n-1},\ker\beta)$. 
The core and the longitude span the page of the open-book structure 
$\{\theta=\textrm{const}\}$ of the Lutz tube. 
Thus we can insert it along any contact knot contactomorphic 
to $(K^{2n-1},\ker\beta)$. We call this insertion a (half) Lutz twist. 
\label{Lutz}
\end{prop}

\begin{rem}
1) Proposition \ref{Lutz} has been generalized to 
a better form in \cite{Massot}. 

2) The twice iteration of half Lutz twist can be considered as 
a full Lutz twist. We notice that this is 
(perhaps essentially) different from the other generalization 
of full Lutz twist found by Etnyre and Pancholi in \cite{Etnyre}. 
Their full twist is much
more elaborated and sophisticated than the original full Lutz twist 
even though phenomenally the former shrinks to the latter in $3$-dimensional case. 
  
\end{rem}

\begin{exmp}
In the case where $n=1$, each connected component of $K^1$ is the circle $S^1$
oriented by a non-zero $1$-form $\beta>0$. Then $\mu=-\beta$ satisfies 
the above conditions. The contact form 
$\alpha'=-\cos(\pi \rho)\beta-\sin(\pi \rho)d\theta$ defines the original 
half Lutz tube $\{\rho\le 1\}=S^1\times D^2$. 
(The full Lutz tube is formally $\{\rho\le 2\}$.)
\end{exmp}

\begin{exmp}
Suppose that $TK^3$ ($n=2$) admits a frame $(e_1, e_2, e_3)$ 
with
$$
[e_3, e_2]=e_1,\quad [e_3, e_1]=e_2 \quad\mathrm{and}\quad [e_1,e_2]=0,
$$
that is, $K^3$ is a $Sol$-manifold. 
Then the dual coframe $(\beta,\mu,\psi)$ satisfies
$$
d\beta=\mu\wedge\psi,\quad d\mu=\beta\wedge\psi
\quad\mathrm{and}\quad d\psi=0
$$
Then we see that $\beta$ and $\mu$ satisfies the above conditions.
For the $Sol$-type contact submanifolds in Theorem \ref{main} (or \ref{Kas}), 
we may put
$$
e_1=\frac{1}{2}(a^zv_-+a^{-z}v_+),\quad
e_2=\frac{1}{2}(a^zv_--a^{-z}v_+)\quad\mathrm{and}\quad
e_3=\frac{1}{\log a}\frac{\partial}{\partial z},
$$
in the setting of Example \ref{1.2}. Then we have $\beta=\beta_++\beta_-$ 
and $\mu=\beta_+-\beta_-$.
\label{Solexmp}
\end{exmp}

\begin{exmp} 
Suppose that $TK^3$ ($n=2$) admits a frame $(e_1, e_2, e_3)$ 
with
$$
[e_3, e_2]=e_1,\quad [e_3, e_1]=e_2 \quad\mathrm{and}\quad [e_1,e_2]=e_3,
$$
that is, $K^3$ is a $\widetilde{SL}_2(\R)$-manifold. 
Since the dual coframe $(\beta,\mu,\psi)$ satisfies
$$
d\beta=\mu\wedge\psi,\quad d\mu=\beta\wedge\psi
\quad\mathrm{and}\quad d\psi=\mu\wedge\beta,
$$
we can see that $\beta$ and $\mu$ satisfies the above conditions.
We also have examples of $\widetilde{SL}_2(\R)$-type submanifolds 
in $S^5$. However, since they lack relation with 
codimension one foliations of $S^5$, 
we omit them in this article. 
\end{exmp}

\begin{exmp}
As is shown in Geiges\cite{Geiges}, there is also 
a $(2n+1)$-dimensional solvable Lie group 
for any $n$ which admits 
a pair of left invariant $1$-forms $\beta$ and $\mu$ 
satisfying the above conditions.
Moreover he found an explicit $T^4$-bundle over the circle
which admits such $\beta$ and $\mu$.
That is, we have an example of a seven dimensional Lutz tube.
The author suspects that seven dimensional Lutz twists enable us 
to change not only the contact structure but also 
the homotopy class of the almost contact structure
of a given contact $7$-manifold. 
See Question 5.5 in Etnyre-Pancholi \cite{Etnyre}. 
(See \cite{Kasuya} snd \cite{Massot} for subsequent developments.)
\end{exmp}

\begin{rem} Once we can perform a Lutz twist, we can iterate 
it any number of times by switching the roles of $\beta$ and $\mu$. 
We can also generalize Giroux twist to the insertion 
of even number of unbinded Lutz tubes $I\times S^1\times K^{2n-1}$ 
along a family $S^1\times K^{2n-1}$ of twistable knots. 
Note that the original Giroux twist is the insertion of
the toric annulus $[0, 2m] \times T^2 \ni (\rho, \theta, \varphi)$ with 
the ``propeller'' contact structure defined by  
$\ker(-\cos(\pi \rho)\varphi-\sin(\pi \rho)d\theta)$  
along a pre-Lagrangian torus in $M^3$. 
\end{rem}

\subsection{Plastikstufes in Lutz tubes}
In this subsection, we show that 
the Lutz tubes of Example \ref{Solexmp} contain plastikstufes. 
First we fix the model.

\begin{defn}
Let $(M^{2n+1},\ker\alpha)$ be a contact manifold, $L^{n-1}$ a closed manifold, 
and $\iota : D^2 \times L^{n-1}\to M^{2n+1}$ an embedding such that the restriction 
$\iota|(\partial D^2 \times L^{n-1})$ is Legendrian. 
Then the image $\iota(D^2\times L^{n-1})$  
is called a plastikstufe if 
$$
\iota^\ast(\alpha)\wedge\{f(\rho)d\rho -\rho d\theta\}= 0,\quad 
\lim_{\rho \to 0}\frac{f(\rho)}{\rho}=0\quad\textrm{and}
\quad\lim_{\rho \to 1}|f(\rho)|=\infty
$$
holds for some functions $h$ and $f(\rho)$ ($(\sqrt{\rho},\theta)\in D^2$). 
See Figure \ref{OF}. 
\label{plastikstufe}
\end{defn}

\begin{figure}[ht]
\centering
\includegraphics{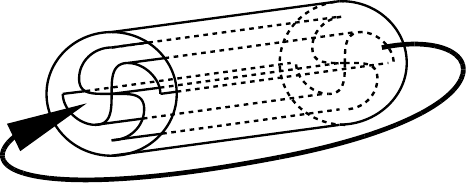}
\caption{A plastikstufe $D^2\times S^1$ in $(M^5,\ker\alpha)$.}
\label{OF}
\end{figure}

Niederkr\"uger and Chekanov introduced this notion and proved 

\begin{thm}[\cite{Niederkruger}]
The contact-type boundary of a compact semipositive symplectic manifold
contains no plastikstufe.
\label{PSOT}
\end{thm}

Then we prove

\begin{thm}
Each of the Lutz tubes given in the above Example \ref{Solexmp}
contains a plastikstufe. Thus we can perform a Lutz twist along each of 
the $Sol$-type contact submanifolds of $S^5$ realized as the singularity links 
in Theorem \ref{main} (or \ref{Kas}) to obtain an exotic contact structure of $S^5$. 
\label{exotic}
\end{thm}

\begin{rem}
This theorem has been generalized and improved to a satisfactory form 
in \cite{Massot}. The next proof is now just an explicit example.
\end{rem}

\proof
In the tube $T_A\times D^2\ni ((p, q, z), (\sqrt{\rho}, \theta))$, with the contact form
$$
\alpha'=\beta_+-\cos(\pi \rho)\beta_--\sin(\pi \rho)d\theta
= a^{-z}dq-\cos(\pi \rho)a^zdp-\sin(\pi \rho)d\theta,
$$
we can take the plastikstufe 
$$
P=\{p=\e a^{-z}g(\pi\rho), q=-\e a^z \cos(\pi \rho)g(\pi\rho)\}
\subset T_A\times D^2
$$
where $\e>0$ is a small constant and $g(\pi\rho)$ a function of $\pi\rho$ such that 
$$
g(\pi\rho)=\frac{1}{\sqrt{\cos(\pi\rho)}}\,\,\,(\pi\rho\ll1) 
\textrm{\quad and \quad} 
g(\pi\rho)=\sin(\pi\rho)\,\,\,(\pi\rho\approx\pi).
$$ 
Indeed the restriction of $\alpha'$ to $P$ is
\begin{eqnarray*}
\alpha' | P &=& \e\pi\{ \sin(\pi \rho)g(\pi\rho)
-2\e\cos(\pi \rho)g'(\pi \rho)\}d\rho-\sin(\pi \rho)d\theta\\
&=&
\frac{\sin(\pi\rho)}{\rho}\left\{\e(\pi\rho g(\pi\rho)-2\pi\rho \cot(\pi\rho)g'(\pi\rho))d\rho-\rho d\theta\right\}
\end{eqnarray*}
Then $f(\rho)=\e(\pi\rho g(\pi\rho)-2\pi\rho \cot(\pi\rho)g'(\pi\rho))$ 
satisfies 
$$
\frac{f(\rho)}{\rho}=0\,\,\,(\rho\ll1)
\textrm{\quad and \quad}
\lim_{\rho\to 1}f(\rho)
=\lim_{\rho\to 1}\left(3\sin(\pi\rho)-\frac{2}{\sin(\pi\rho)}\right)
=-\infty.
$$
\qed
\begin{rem}
As $\e\to 0$, the above plastikstufe converges 
to a solid torus $S^1\times D^2$ 
foliated by $S^1$ times the straight rays on $D^2$, i.e., the leaves are 
$\{\theta=\mathrm{const}\}$. Note that this solid torus is the preimage of 
the closed orbit $\{p=0,q=0\}$ of the suspension 
Anosov flow $((x,y),z)\mapsto((x,y),z+t)$ 
on the core $T_{A^{-1}}$ under the natural projection. 
This is called an overtwisted family in \cite{Etnyre}. 
\end{rem}

\section{Violation of the inequality and generalized Reeb components}

\subsection{Milnor fibrations as supporting open-book structures}

We can see that the exotic contact structure of $S^5$ constructed 
in Theorem \ref{exotic} is convex-overtwisted 
from the following lemma essentially contained in the note 
\cite{GirouxNote} of Giroux.

\begin{lem}[see \cite{GirouxNote}]
The Milnor tube of an isolated singularity $(0,\dots,0)\in\C^{n+1}$ of complex hypersurface
determines an isotopy class of contact structures of $S^{2n+1}$ via 
the exact symplectic open-book structure associated to it.
Moreover the isotopy class is represented by the standard contact structure of $S^{2n+1}$.
\label{sfl}
\end{lem}

\begin{proof}[Sketch of the proof]
On the small ball $B_\e=\{|z_0|^2+\dots+|z_{n+1}|^2\le \e^2\}$
in $\C^{n+2}$, we consider the graph $G_k$ of the 
function $z_{n+1}=kf(z_0,\dots,z_n)$.
From the Gray stability, the contact structure of $G_k$ 
is isotopic to that of the standard $S^{2n+1}=\Gamma_0$. 
Writing $z_{n+1}$ as $x_{n+1}+\sqrt{-1}y_{n+1}$, we see from the obvious inequality  
$$
(-y_{n+1}dx_{n+1}+x_{n+1}dy_{n+1})
(-y_{n+1}\partial/\partial x_{i+1}+x_{n+1}\partial/\partial y_{n+1})\ge 0,
$$
and $dz_{n+1}|\Sigma_\infty=0$, that  $\arg(f|\partial \Sigma_k)$ 
defines a supporting open-book structure of $\partial G_k$ 
equivalent to the Milnor tube 
$\{\{f=\delta\}\cap B_\e\}_{|\delta|=\e'}$ 
if $k$ is sufficiently large and $\e'>0$ is sufficiently small. 
Indeed, we can construct a contact vector field on $\partial G_k$ 
which is close to the rotation vector field 
$-y_{n+1}\partial/\partial x_{i+1}+x_{n+1}\partial/\partial y_{n+1}$ 
and transverse to 
the contact structure as well as to the pages $\{\arg f|G_k=\textrm{const}\}$.
\end{proof}

\begin{thm}
1) Suppose that the Euler characteristic of the page 
$\{\theta=\textrm{const}\}$ of an exact symplectic open-book structure
$\theta:M^{2n+1}\setminus N^{2n-1}\to S^1$ is positive, and 
each connected component of the contact submanifold $N^{2n-1}$ 
is twistable. Then we can insert the Lutz tube
along $N^{2n-1}$ to obtain a new contact structure
which is convex-overtwisted. 
Indeed since we reversed the orientation of pages and added 
positive collars, the pages are convex overtwisted hypersurfaces. 

2) Especially each exotic contact structure of $S^5$ constructed in 
Theorem \ref{exotic} is convex-overtwisted. 
(The Milnor fiber is homotopically a bouquet of $2$-spheres.)
\end{thm}

\subsection{Generalized Reeb components}

Putting $\nu=d\psi$ in Example \ref{Solexmp}, 
we see that the open-book structures associated 
to the Milnor fibrations of the singularities in 
Theorem \ref{main} (or Theorem \ref{Kas}) 
satisfy the condition of Theorem \ref{convergence}.
Let $\F$ denote the limit foliation of a deformation of 
the standard contact structure $\ker\alpha$ associated 
to a $Sol$-type contact submanifold $N^3$.   

\begin{thm}
Let $\ker\alpha'$ be the exotic convex-overtwisted 
contact structure obtained by inserting the Lutz 
tube with core $(-N^3,\mu)$ along the above 
$Sol$-type submanifold $(N^3,\beta)\subset(S^5,\ker\alpha)$.
Then we can deform the exotic contact structure $\ker\alpha$ 
via contact structures to a foliation $\F '$ with 
two compact leaves which are parallel to each other. 
The foliation $\F'$ can be obtained by 
cutting and turbulizing the above foliation $\F$ 
along the boundary $\partial U$ of a regular neighborhood $U$ 
of the Reeb component of $\F$. Then 
$\partial U\approx N^3\times S^1$ becomes
a compact leaf.
\label{turbulization}
\end{thm}

\noindent
We can prove this theorem by combining Theorem \ref{convergence} with 
the following convergence of contact structures to generalized Reeb components.

\begin{thm} (Convergence to generalized Reeb components)
Let $\widetilde{\alpha}$ be the contact form of 
the unified contactization 
$(\Sigma_+\cup(-\Sigma_-))\times\R$ in Definition \ref{unified}. 
Take a diffeomorphism $\varphi$ and a positive function $h$ 
which satisfies the conditions described in Theorem \ref{openbook2} 
except that in this case $\partial\Sigma=\emptyset$. 
Then the mapping torus $M^{2n+1}=\Sigma\times\R/(x,z+h)\sim(\varphi(x),z)$ 
possesses the contact form $\alpha_0$ induced from $\widetilde{\alpha}$.
Then there exists a family 
$\{\alpha_t\}_{0\le t<1}$ of contact forms which starts from $\alpha_0$ 
and converges to a defining $1$-form $\alpha_1$ of a codimension one foliation $\F$. 
Here $\F$ coincides with the level foliation of $h$ except near the compact leaves
$\Gamma\times\R/\Z$ into which the level foliation spiral.  
Moreover, we may assume that the compact leaves decomposes 
the foliation into the union of dead-end components, i.e., generalized Reeb components. 
\end{thm}

\proof
On a neighborhood 
$(-\e',\e')\times \Gamma\times \R/\Z (\subset M^{2n+1})$ 
of the hypersurface $\Gamma\times \R/\Z$, 
the contact form $\alpha_0$ is expressed as
$$
\displaystyle\alpha_0=e^{-s^2}\beta+g(s)dz
\quad (s\in (-\e',\e')),
$$
where $\beta$ is the contact form on $\Gamma$ and 
$g(s)$ is a decreasing function which coincides with $-s$ near $s=0$
and which is smoothly tangent to $\mp 1$ at $s=\pm\e'$. 

Put $\tau=(1-t)^2$ and take 
a function $e(s)$ supported in $(-\e',\e')$ with $e(0)>0$.
Then the family of contact forms
$$
\alpha_t=\tau\alpha_0+(1-\tau)g(s)dz+(1-\tau)e(s)ds\quad (0\le t<1)
$$
converges to $\alpha_1$which defines the foliation $\F$ described in the theorem. \qed

\begin{rem}
We can change simultaneously the orientation of 
the compact leaves by changing the sign of the value $e(0)$ totally. 
However in order to obtain the dead-end components we
can not change it partially.  
\end{rem}

\subsection{Topology of the pages}
In this subsection, we decide the Euler characteristic of the Milnor fiber 
of each singularity in Theorem \ref{main}. Note that this is equal to 
that of the corresponding singularity in Theorem \ref{Kas} 
which is well known (see \cite{Kasuya}). Thus the followings calculation 
has become just an alternative approach.   
  
The Milnor fiber is diffeomorphic to
$$F=\{f_{m,k}(\xi,\eta,\zeta)=\delta\}\cap 
\{|\xi|^2+|\eta|^2+|\zeta|^2\le \e\},$$ 
where $\delta \in\C$ and $0<|\delta|\ll \e \ll 1$. 
Note that the Euler characteristic $\chi(F)$ is equal to $\mu(f_{m,k},(0,0,0))+1$ 
where $\mu()$ denotes the Milnor number of the function germ.
Let $\pi_\xi,\pi_\eta$ and $\pi_\zeta$ 
denote the projections to the axes.

In the case where $m=1$, the critical values of $\pi_\zeta|F$ are 
the solutions of the following system:
$$f_{1,(k_1)}-\delta=0,\quad
\frac{\partial}{\partial \xi} f_{1,(k_1)}=2\xi=0,\quad 
\frac{\partial}{\partial \eta} f_{1,(k_1)}=0,\quad
\textrm{and} \quad |\zeta|\ll \e.$$ 
Therefore, for each critical value $\zeta$ of $\pi_\zeta|F$, 
we have the factorization
$$(\eta-2\zeta^2)(\eta^2+2\zeta^2\eta+\zeta^4-\zeta^{4+k_1})-\delta = (\eta-a)^2(\eta+2a)$$
of the polynomial of $\eta$, where the parameter $a \in\C$ depends on $\zeta$. 
By comparison of the coefficients of the $\eta^1$-terms and the $\eta^0$-terms we have
$$-4\zeta^4+\zeta^4-\zeta^{4+k_1}=a^2-2a^2-2a^2
\quad \textrm{and}\quad -2\zeta^6+2\zeta^{6+k_1}-\delta=2a^3.$$
Eliminating the parameter $a$, we obtain the equation
$$4\zeta^{12+k_1}(9-\zeta^{k_1})^2=
108\zeta^6(1-\zeta^{k_1})\delta+27\delta^2.$$
Then we see that $\pi_\zeta|F$ has $12+k_1$ critical points, 
which indeed satisfy $a\ne -2a$, i.e., 
the map $\pi_\zeta|F$ defines a Lefschetz fibration $F\to D^2$ 
with $12+k_1$ singular fibers. 
Thus we have $$\chi(F)=1-2+12+k_1=11+k_1.$$

In the case where $m=2$, we have the factorization
$$\{(\zeta+\eta)^2-\zeta^{2+k_1}\}\{(\zeta-\eta)^2+\zeta^{2+k_2}\}-\delta = 
(\eta-a)^2(\eta+a-b)(\eta+a+b).$$
By comparison of the coefficients we have
$$
\left\{
\begin{array}{l}
\zeta^2(2+\zeta^{k_1}-\zeta^{k_2})=2a^2+b^2\\
\zeta^3(\zeta^{k_1}+\zeta^{k_2})=ab^2\\
\zeta^4(1-\zeta^{k_1})(1+\zeta^{k_2})-\delta=a^2(a^2-b^2)
\end{array}
\right..
$$
In order to eliminate $a,b$, we put $a=u+v$ and $\zeta^2(2+\zeta^{k_1}-\zeta^{k_2})=6uv$. 
Then we have
$$
\left\{
\begin{array}{l}
6uv-2(u+v)^2=b^2\\
\zeta^3(\zeta^{k_1}+\zeta^{k_2})=-2(u^3+v^3)\\
\zeta^4(1-\zeta^{k_1})(1+\zeta^{k_2})-\delta=(u+v)^4+2(u^3+v^3)(u+v)
\end{array}
\right..
$$
Following Cardano's method, we put 
$$p=uv,\quad q=u^3+:v^3\quad \textrm{and}\quad 
r=(u+v)^4+2(u^3+v^3)(u+v).$$
Then $p,q$ and $r$ are polynomials of $\zeta$. 
Eliminating $a$ from 
$$q(=q(p,a))=3pa-a^3\quad \textrm{and}\quad r(=r(p,a))=-6pa^2+3a^4,$$
we obtain 
$$
(27q^4-r^3)+54(prq^2-p^3q^2)+18p^2r^2-81p^4r=0,
$$
which is a polynomial equation of $\zeta$. As $\delta\to0$, the left hand side converges to
$$
\zeta^{12+k_1+k_2}
\left\{1-\frac{\zeta^{k_1}-\zeta^{k_2}}{2}+\frac{(\zeta^{k_1}+\zeta^{k_2})^2}{16}
\right\}^2.
$$
Therefore $\pi_\zeta|F$ has $12+k_1+k_2$ critical points, 
which indeed satisfy $4a^2\ne b^2$ and $b\ne 0$, i.e., 
the map $\pi_\zeta|F$ defines a Lefschetz fibration $F\to D^2$ 
with $12+k_1+k_2$ singular fibers. 
Thus we have $$\chi(F)=1-3+12+k_1+k_2=10+k_1+k_2.$$

\subsection{Open problems}
At present, the following natural questions are open. 

\begin{pbm}
1) Is there any convex-tight contact structure of dimension $>3$ ?

2) When is a Lutz twisted contact structure convex-overtwisted ?

3) Are the next to standard contact structures of $S^{2n+1}$ convex-overtwisted ? 

4) Is there any relation between our half Lutz twist and
the full Lutz twist of Etnyre and Pancholi ? 
\label{Q}
\end{pbm}

The next problem can be considered as a variation of Calabi's question.

\begin{pbm}
Does the standard $S^{2n+1}$ ($n>1$) contains a convex hypersurface 
with disconnected contact-type boundary?
\end{pbm}

\noindent
If there is no such hypersurfaces, the following conjecture trivially holds.

\begin{conj}[As an affirmative answer for the above 1)] 
$S^{2n+1}$ is convex-tight. 
\end{conj}

\section*{Acknowledgement}
The author would like to cordially thank John Etnyre, Naohiko Kasuya 
and Klaus Niederkr\"uger for encouraging the author who 
had almost gave up publishing the results of this article.
He would also like to thank Yoshihiko Mitsumatsu for helpful comments
especially for giving an idea on Lutz twists along 
Brieskorn $\widetilde{SL}_2(\R)$-type contact submanifold 
which will be discussed in another place.

\section*{Appendix: On the violation of Thurston-Bennequin inequality 
for a certain non-convex hypersurface}

In this short note, we give an example of  arbitrary small hypersurface 
with contact-type boundary in a Darboux chart of dimension greater than three 
which violates the Thurston-Bennequin inequality. We also confirm 
its non-convexity. 

Let $(r,\theta,z)$ be the cylindrical coordinates of $\R^3$, and take the functions
$$
\lambda(r)=2r^2-1 \quad 
\mathrm{and} \quad
\mu(r)=r^2(r^2-1).
$$
Then the contact form
$$
\beta=\lambda(r)dz+\mu(r)d\theta
$$
defines a contact structure with the overtwisted disk $D_{r\le 1}^2\times \{0\}$. 
Let $U$ denote a small neighborhood of $D_{r\le 1}^2\times\{0\}$.  
Since even a $3$-ball in a Darboux chart (shortly a Darboux $3$-ball) 
contains immersed overtwisted disks, it is easy to see that a Darboux $5$-ball 
contains an embedded overtwisted disk. 
Thus we can embed the product 
$U\times (D_{<\e}^2)^{n-1}$ equipped with the contact form 
$\displaystyle \alpha=\beta+\sum_{i=1}^{n-1}\left(x_idy_i-y_idx_i\right)$   
into a Darboux $(2n+1)$-ball, where $D_{<\e}^2=\{x_i^2+y_i^2<\e^2\}$ 
are small disks on the $x_i+\sqrt{-1}y_i$-axes. (Note that 
$\displaystyle f^2\sum_{i=1}^{n-1}(x_idy_i-y_idx_i)
=\sum_{i=1}^{n-1}(fx_id(fy_i)-fy_id(fx_i))$ 
holds for any function $f$.) Now we take the hypersurfaces
$$
\widetilde{\Sigma}=\left\{
r^2+\e^{-2}\left(z^2+\sum_{i=1}^{n-1}(x_i^2+y_i^2)\right)=1+\e
\right\}\quad \textrm{and}\quad
\Sigma=\left\{r-z\le 1\right\}\cap \widetilde{\Sigma}.
$$ 
The characteristic foliation $\F_{\widetilde{\Sigma}}$ 
is presented by the vector field
\begin{eqnarray*}
X
&=& \displaystyle
\e^{-2}r(r^2-1)z\partial_r
+(1+2\e-2\e^{-2}z^2)\partial_\theta
\\
&& \displaystyle
+\left\{
(r^2-1)^2+(2r^2-1)(\e^{-2}z^2-\e)
\right\}\partial_z
\\
&& \displaystyle
+\,\e^{-2}(2r^2-1)z\sum_{i=1}^{n-1}
\left(
x_i\partial_{x_i}+y_i\partial_{y_i}
\right)
\\
&& \displaystyle
+\e^{-2}(2r^4-2r^2+1)\sum_{i=1}^{n-1}
\left(
-y_i\partial_{x_i}+x_i\partial_{y_i}
\right).
\end{eqnarray*}
Indeed the following calculations shows that 
the vector field $X$ satisfies $X \in T\widetilde{\Sigma}$,
$(\alpha|T\widetilde{\Sigma})(X)=0$, and 
$\mathcal{L}_X (\alpha|T\widetilde{\Sigma})
=2\e^{-2}(2r^2-1)z\alpha|T\widetilde{\Sigma}$.
\begin{eqnarray*}
&&\hspace{-20mm}
\displaystyle
\left\{
2rdr+\e^{-2}\left(2zdz+2\sum_{i=1}^{n-1}(x_idx_i+y_idy_i)\right)
\right\}
(X)
\\
&=&\displaystyle
2\e^{-2}(2r^2-1)z
\left\{
r^2+\e^{-2}
\left(
z^2+\sum_{i=1}^{n-1}(x_i^2+y_i^2)
\right)-1-\e
\right\},
\\
\alpha
&=& \displaystyle
(2r^2-1)dz+r^2(r^2-1)d\theta+\sum_{i=1}^{n-1}\left(x_idy_i-y_idx_i\right),
\\
\alpha(X)
&=& \displaystyle
(2r^2-1)\left\{
(r^2-1)^2+(2r^2-1)(\e^{-2}z^2-\e)
\right\}
\\ 
&& \displaystyle
+r^2(r^2-1)\{1-2(\e^{-2}z^2-\e)\}+\e^{-2}(2r^4-2r^2+1)\sum_{i=1}^{n-1}(x_i^2+y_i^2)
\\ 
&& \displaystyle
=(2r^4-2r^2+1)
\left\{
r^2+\e^{-2}
\left(
z^2+\sum_{i=1}^{n-1}(x_i^2+y_i^2)
\right)-1-\e
\right\},\\
d\alpha 
&=& \displaystyle
4rdr\wedge dz+2r(2r^2-1)dr\wedge d\theta+2\sum_{i=1}^{n-1}dx_i\wedge dy_i,
\\
\iota_Xd\alpha
&=& \displaystyle 
\e^{-2}r(r^2-1)z(4rdz+2r(2r^2-1)d\theta)-2r(2r^4-2r^2+1)dr
\\
&& \displaystyle
+2\e^{-2}(2r^2-1)z\sum_{i=1}^{n-1}(x_idy_i-y_idx_i)
\\
&& \displaystyle
-2\e^{-2}(2r^4-2r^2+1)\sum_{i=1}^{n-1}(x_idx_i+y_idy_i)
\\ 
&=& \displaystyle
2\e^{-2}(2r^2-1)z\alpha \\
&& \displaystyle
-(2r^4-2r^2+1)
\left\{
2rdr+\e^{-2}
\left(
2zdz+2\sum_{i=1}^{n-1}(x_idx_i+y_idy_i)
\right)
\right\}.
\end{eqnarray*}
Further we see that the vector field $X$ satisfies
$$
(dr-dz)(X)|\partial\Sigma=(r-1)^2\{\e^{-2}(-r^2+r+1)-(r+1)^2\}+\e(2r^2-1)
>0
$$
with attention to $z=r-1$. 
Thus $\partial \Sigma$ is a contact-type boundary.
 
The singularity of $X|\Sigma$ is the union of the one point set
$$
S_+(\Sigma)=\{((0,\theta,-\e\sqrt{1+\e}),0) \in U\times B_{<\e}^{2n-2}\}
$$
and the other one point set 
$$
S_-(\Sigma)=\{((0,\theta,+\e\sqrt{1+\e}),0) \in U\times B_{<\e}^{2n-2}\}.
$$
They are respectively a source point and a sink point.
Since the indices of these points are equal to $1$, 
we see that the hypersurface $\Sigma$ violates the Thurston-Bennequin inequality.
Figure A depicts (the four-fold covering of)
the well-defined push-forward $X'$ of $X$ under the natural projection $p$ 
from $\widetilde{\Sigma}$ to the quarter-sphere 
$$
\Sigma'=\left\{(z,r,|(x,y)|) \mid r^2+\e^{-2}(z^2+|(x,y)|^2)=1+\e\right\} \quad (r\ge 0, \, |(x,y)|\ge 0).
$$

\begin{figure}[h]
\centering
\includegraphics{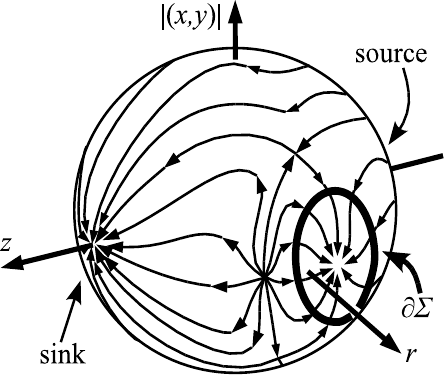}
\begin{center}
{\bf Fifure A.}\quad The covering of the vector field $X'$
\end{center}
\end{figure}

The vector field $X'$ on the quoter $2$-sphere defines the singular foliation
$$
\F'=\{\e^{-2}z^2=(Cr^2-1)(r^2-1)+\e\}_{-\infty \le C \le +\infty}.
$$
The singularity consists of the following five points; two (quarter-)elliptic points
$\left(\mp\e\sqrt{1+\e},0,0\right)$, whose preimages under $p$ are the above singular points;
other two (half-)elliptic points $\left(\pm\e\sqrt{\e},1,0\right)$, whose preimages are the periodic orbits 
$P_\pm$ of $X$; 
and a single hyperbolic point 
$\left(0,\sqrt{1+\e-\sqrt{\e(1+\e)}},\sqrt{\e^2\sqrt{\e(1+\e)}}\right)$
presenting the double point of the singular level $C=1+2\e+2\sqrt{\e(1+\e)}$.
Slightly changing the small positive constant $\e$ if necessary, 
we may assume that the preimage $H$
of this hyperbolic point ($\approx S^1\times S^{2n-1}$) is the union of 
periodic orbits of $X$. 

Now we assume that $\widetilde{\Sigma}$ is convex. 
Then it is divided along a contact submanifold $\Gamma$ 
into the positive region $\widetilde{\Sigma}_+$ 
and the negative region $\widetilde{\Sigma}_-$. 
Moreover $X$ is transverse to $\Gamma$ from positive region to negative region,
and there the sign of the natural divergence of $X$ changes from positive to negative.
This implies that 
$$
S_i(\Sigma), \, P_i \subset \widetilde{\Sigma} \quad (i=+,-)
\quad\textrm{and}\quad
\Gamma \cap H=\emptyset.
$$
(Note that if $\widetilde \Sigma$ itself is not convex but becomes convex 
after a small perturbation, the same properties also hold.)  
Then we can see that the dividing set $\Gamma$ must contain 
a spherical component. This contradicts to 
the famous Eliashberg-Floer-McDuff theorem, 
which says that $S^{2n-1}\coprod$(other components) 
can not be realized as the contact-type boundary 
of a connected symplectic manifold. 
This proves that $\widetilde{\Sigma}$ is not convex.  
Similarly we can prove the non-convexity of $\Sigma$. 
\end{document}